\documentclass[12pt]{amsart}

\usepackage{amsmath,amsthm,amssymb}

\usepackage{enumitem}

\makeatletter
\@namedef{subjclassname@2020}{%
  \textup{2020} Mathematics Subject Classification}
\makeatother

\newtheorem{theorem}{Theorem}[section]
\newtheorem{corollary}[theorem]{Corollary}
\newtheorem{lemma}[theorem]{Lemma}
\newtheorem{proposition}[theorem]{Proposition}
\newtheorem{conjecture}[theorem]{Conjecture}

\theoremstyle{definition}

\newtheorem{remark}[theorem]{Remark}
\newtheorem{example}[theorem]{Example}

\numberwithin{equation}{section}

\newcommand{\beq}{\begin{equation}}  
\newcommand{\eeq}{\end{equation}}  
\newcommand{\bear}{\begin{array}}  
\newcommand{\eear}{\end{array}} 
\newcommand\la{{\lambda}}   
\newcommand\La{{\Lambda}}  
\newcommand\al{{\alpha}}   
\newcommand\bet{{\beta}}   
\newcommand\gam{{\gamma}}   
\newcommand\eps{{\epsilon}}    
\newcommand{\Z}{\mathbb{Z}}

\newcommand{\Zp}{\mathbb{Z}_{>0}}
\newcommand{\Q}{\mathbb{Q}}

\newcommand{\R}{\mathbb{R}}

\begin{document}
%----------------

%%%%% To ease editing, for IMPAN journals add:

\baselineskip=17pt

\title[Continued fractions for strong Engel series with signs]{Continued fractions for strong Engel series 
and L\"{u}roth series with signs} 
\author[A. N. W. Hone]{Andrew N. W. Hone}
\address{School of Mathematics, Statistics and Actuarial Science\\  University of Kent\\ 
Canterbury CT2 7FS, UK}
\email{A.N.W.Hone@kent.ac.uk}

\author[J. L. Varona]{Juan Luis Varona}
\address{Departamento de Matem\'aticas y Computaci\'on\\  
Universidad de La Rioja\\ 26006 Logro\~no, Spain}
\email{jvarona@unirioja.es}

\date{} 


\begin{abstract} 
An Engel series is a sum of 
reciprocals $\sum_{j\geq 1} 1/x_j$ of a non-decreasing sequence of positive integers $x_n$
with the property that $x_n$ divides $x_{n+1}$ for all $n\geq 1$. In previous work, 
we have shown that for any Engel series with the stronger property 
that $x_n^2$ divides $x_{n+1}$, the continued fraction expansion of the sum 
is determined explicitly in terms of $z_1=x_1$ and the ratios $z_n=x_n/x_{n-1}^2$ for $n\geq 2$. Here we 
show that, when this stronger property holds, 
the same is true for a sum $\sum_{j\geq 1}\eps_j/x_j$ 
with an arbitrary 
sequence of signs $\eps_j=\pm 1$. 
As an application, we use this result to provide explicit continued fractions for 
particular families of L\"{u}roth series and alternating L\"{u}roth series defined by nonlinear recurrences of second order.  
We also calculate exact irrationality exponents for certain families of transcendental numbers 
defined by such series. 
\end{abstract}



%----------------
\maketitle 
%----------------

%----------------
\section{Introduction} %and main results}
\label{sec:intro}
%----------------

Given a sequence of positive integers $(x_n)$, which is such that $x_n \mid x_{n+1}$ 
 for all $n\geq1$, the sum of the reciprocals is the Engel series
\begin{equation}
\label{eq:Engel}
 S = \sum_{j=1}^{\infty} \frac{1}{x_j}
 %= \sum_{j=1}^{\infty} \frac{1}{y_1y_2 \cdots y_j},
\end{equation}
%where $y_1 = x_1$ and $y_{n+1} = x_{n+1}/x_n$ for $n \ge 1$,
and the alternating sum of the reciprocals is the Pierce series
\begin{equation}
\label{eq:Pierce}
  S' = \sum_{j=1}^{\infty} \frac{(-1)^{j+1}}{x_j}.
  %= \sum_{j=1}^{\infty} \frac{(-1)^{j+1}}{y_1y_2 \cdots y_j}.
\end{equation}
(To ensure convergence it should be assumed that $(x_n)$ is eventually increasing, i.e.\ 
%sense that 
for all $n$ there is some $n'>n$ with $x_{n'}>x_n$.) 
%Given any real number, admits an Engel expansion 
Every positive real number admits both an Engel expansion, of the 
form \eqref{eq:Engel}, and a Pierce expansion \eqref{eq:Pierce} \cite{duverney, erdosshallit}, and these 
$S,S'$ are unique: 
after removing the integer part it is sufficient to consider numbers in the interval $(0,1)$, and then $(x_n)$ is strictly increasing with $x_1\geq 2$ 
in (\ref{eq:Engel}) and $x_1\geq 1$ in (\ref{eq:Pierce}). Although they are not quite so well 
known, Engel expansions and Pierce expansions have much in common with 
continued fraction expansions, both in the way that they are determined recursively, and from a 
metrical point of view; for instance, see \cite{erdosetal} for the case of Engel series and \cite{shallitPierce} for Pierce series. 

In recent work \cite{Hone1}, 
the first author presented a family of sequences $(x_n)$ generated by certain nonlinear recurrences of second order, of the form 
\beq\label{2nd} 
x_{n+2}x_{n} = x_{n+1}^2 \,\big(1+x_{n+1}\,G(x_{n+1})\big), \qquad n\geq 1, \qquad G(x)\in \Z [x], 
\eeq 
where $G(x)>0$ for $x>0$, % takes positive values at positive arguments, 
such 
that the corresponding Engel series~\eqref{eq:Engel} yields a transcendental number 
with a regular 
continued fraction expansion 
$$ 
S = 
  [a_0;a_1,a_2,a_3, \dots] = a_0 + \cfrac{1}{ a_1 + \cfrac{1}{ a_2 + \cfrac{1}{ a_3+ \cdots } } }
$$
whose coefficients (partial quotients) $a_0,a_1,\ldots$ 
are explicitly given in terms of the~$x_n$. 

More recently \cite{Var}, the second 
author proved that, when the sequence $(x_n)$ is generated by the same sort of nonlinear recurrence 
as (\ref{2nd}), 
an analogous result holds for the associated Pierce series~\eqref{eq:Pierce}, 
although the structure of the corresponding continued fractions is different. 
%In fact, in the latter 
%work the polynomial $G(x_n)$ was replaced by an arbitrary sequence of 
%positive integers, as it had already been 
It was subsequently noted in \cite{Hone3, HoneVar} that 
the recurrence 
(\ref{2nd}) could be further generalized, and at the same time allow the 
explicit continued fraction expansion to be determined for the sum 
of an arbitrary rational number $r=p/q$ and an Engel series, that is 
\beq\label{pqengel} 
\frac{p}{q} + \sum_{j=2}^\infty \frac{1}{x_j}, \qquad 
\text{with}\,\,\, q=x_1 \mid x_2, 
\eeq  
and similarly for the case of $p/q$ $\pm$ a Pierce series.
 
The key to the results in \cite{Hone1, Hone3, HoneVar} and \cite{Var} was that, 
subject to a recurrence like (\ref{2nd}), the truncation of the particular series 
(\ref{eq:Engel}) or (\ref{eq:Pierce}) %being considered 
at the $n$th term yields 
a convergent of the continued fraction of $S$ or $S'$, whose length depends linearly on $n$. 
Continuing in this vein, Duverney et al.\ showed in \cite{DKS1} that a finite sum 
\beq\label{yxsum}
\sum_{j=1}^n \frac{y_j}{x_j} 
\eeq 
can be expanded as a continued fraction of general type, of length $2n$, i.e.\
\beq\label{dks} 
 \cfrac{a_1}{ b_1 + \cfrac{a_2}{ b_2 + \cfrac{a_3}{ \cdots+\cfrac{a_{2n}}{b_{2n}} } } }, 
\eeq  
where $a_j,b_j$ are explicit rational functions of the indeterminates $x_j,y_j$, and they presented 
a similar formula for the alternating sum $\sum_{j=1}^n (-1)^{j-1} y_j/x_j$ as a general 
continued fraction of length $3n-4$. 

The formulae for the continued fraction (\ref{dks}) are conveniently 
written in terms of the ``exponentiated shift'' operator $\theta$ from \cite{DKS1}, defined by 
\beq\label{eshift} 
\theta [u_n]=\frac{u_{n+1}}{u_n}, 
\eeq 
and if it is assumed that $(x_n)$ is an increasing sequence of positive integers with
$x_1\geq 2$ and $(y_n)$ is another sequence of positive integers, then  
(writing $\theta^2 [ x_n]=\theta[\theta [ x_n]]$) 
the recurrence 
\beq\label{dksrec} 
\theta^2 [x_n]-\theta^2[y_n]=\alpha_{n} x_n, 
\eeq 
for an arbitrary sequence $(\alpha_n)$ %\in \Zp$ 
consisting of positive integers, provides a natural 
generalization of (\ref{2nd}), with the results on Engel and Pierce series in \cite{Hone1, Hone3, Var} 
corresponding to the special case $y_n=1$ for all~$n$. Moreover, the irrationality exponents of transcendental 
numbers given by suitable Engel and Pierce series were explicitly calculated in \cite{HoneVar}; and 
in \cite{DKS2} (based on a result from \cite{DS}) 
this was further extended to find the irrationality exponents of the limits 
$n\to\infty$ for more general series (\ref{yxsum}), subject to the recurrence 
(\ref{dksrec}) with appropriate assumptions on the growth of the sequences 
$(x_n)$, $(y_n)$ and~$(\al_n)$. 
 
In a separate development \cite{Hone2}, the first author showed that if the denominators 
in an Engel series satisfy the stronger divisibility property 
\beq\label{strong} 
x_n^2 \mid x_{n+1}, \qquad n\geq 1, 
\eeq 
then the continued fraction expansion of (\ref{eq:Engel}) 
can be written explicitly in terms of 
the integers $z_j$ defined by 
\beq\label{zdef}
z_1=x_1, \qquad z_{j+1}=\frac{x_{j+1}}{x_j^2} \in \Zp, \quad j\geq 1. 
\eeq 
Henceforth we refer to a series (\ref{eq:Engel}) with the property (\ref{strong}) as a 
strong Engel series. 

Included in this class of strong Engel series is the set of numbers 
\beq\label{kempner} 
\sum_{n=0}^\infty \frac{1}{u^{2^n}} 
\eeq 
for integer $u\geq 2$, with the case $u=2$ being known as the Kempner number \cite{adamc}. All of these numbers 
are transcendental, with irrationality exponent $2$ \cite{bugeaudetal}; 
their 
continued fraction expansions were found in 
recursive form in \cite{shallit1}, with a non-recursive representation 
described in \cite{shallit2}, and further generalizations with a 
similar folded recursive structure were given in \cite{shallit3} and later \cite{folded}. The series that are treated 
in the latter works all depend on a single integer parameter, e.g.\ the integer $u$ in (\ref{kempner}), 
whereas strong Engel series and their generalizations to be considered in the sequel depend 
on the infinite set of parameters $z_j$.

In the next section we generalize the results of \cite{Hone2} to the case of a series 
\beq\label{rndsign} 
\sum_{j=1}^\infty \frac{\eps_j}{x_j}, \qquad \eps_j=\pm 1, 
\eeq 
where the denominators satisfy the strong Engel property (\ref{strong}) and the sequence of signs 
$\eps_j$ is arbitrary. In fact, using the same approach as in \cite{HoneVar} we provide the explicit continued 
fraction expansion for a sum of the form $p/q$ $+$ a strong Engel series. In the third section we give an application of these results 
to a family of L\"uroth series, that is series of the type 
\beq\label{luroth} 
\frac{1}{u_1} +\sum_{j=2}^{\infty} \frac{1}{u_{1}(u_{1}-1) \cdots 
  u_{j-1}(u_{j-1}-1) u_{j}}, 
\eeq 
where we impose the condition that the sequence $(u_n)$ satisfies a nonlinear recurrence of second order analogous to (\ref{2nd}). We also 
consider one of the alternating analogues of L\"uroth series introduced in \cite{kkk}, and 
other generalizations along similar lines. The final section is mostly devoted to calculating exact irrationality exponents for 
certain families of series of generalized L\"uroth type, 
defined by particular recurrences of second order for~$u_n$. 
Inspired by \cite{DKS2}, these recurrences are given in terms of ``pseudo-polynomials'' with arbitrary real exponents 
(see (\ref{alform}) below), 
rather than polynomials like $G$ in (\ref{2nd}). 
We conclude with a  
conjecture about transcendence 
of strong Engel series with signs.

%----------------
\section{Some explicit continued fractions} 
\label{sec:ecf}
%----------------

%Before proceeding, we fix our notation for continued fractions and briefly mention some of their 
%standard properties, which can be found in many books %, for instance, \cite{Cas, duverney, Khin}. 
To fix our notation, we briefly recall some standard facts about continued fractions. 
In what follows, we denote a finite regular continued fraction by 
\beq\label{ficf}
  [a_0;a_1,a_2,\dots,a_n] 
  = a_0 + \cfrac{1}{ a_1 + \cfrac{1}{ a_2 + \cfrac{1}{ \cdots +\cfrac{1}{a_n}} } }
  = \frac{p_n}{q_n},
\eeq
where $a_0 \in \Z$, $a_j \in \Zp$ and the convergent $p_n/q_n$ is in lowest terms with $q_n>0$. 
We define the length of (\ref{ficf}) to be the index of the final partial quotient, 
written as 
$$ 
\ell (  [a_0;a_1,a_2,\dots,a_n] )=n; 
$$ 
so we ignore the integer part $a_0$ when counting the length.
Every $r\in\Q$ 
can be written as a finite continued fraction (\ref{ficf}), although 
this representation is not unique (there is both an odd and an even length representation), 
% (see (\ref{props}) below). 
but each $\xi\in\R\setminus\Q$ 
is given uniquely by an infinite continued fraction with convergents $p_n/q_n$ 
of the form (\ref{ficf}), that is (with $a_0=\left \lfloor{\xi}\right \rfloor $) 
\beq\label{al}
\xi = 
  [a_0;a_1,a_2,\dots] = \lim_{n\to\infty} [a_0;a_1,a_2,\dots,a_n] = \lim_{n\to\infty}\frac{p_n}{q_n}.
\eeq 
%and this limit always exists. 

The three-term recurrence relation satisfied by the numerators and denominators of the convergents 
%can be conveniently 
is equivalent to the matrix relation 
\beq\label{mat} 
\left(\bear{cc} p_{n+1} & p_n \\ 
q_{n+1} & q_n 
\eear\right) = 
\left(\bear{cc} p_n & p_{n-1}\\ 
q_n & q_{n-1}
\eear\right)\left(\bear{cc} a_{n+1} & 1 \\ 
1 & 0
\eear\right), 
\eeq 
for all $n\geq -1$, with
$$ 
\left(\bear{cc} p_{-1} & p_{-2} \\ 
q_{-1} & q_{-2} 
\eear\right) = \left(\bear{cc} 1 & 0 \\ 
0 & 1
\eear\right).
$$
Taking determinants in (\ref{mat}) yields the standard 
identity 
\beq\label{detid} 
  p_j q_{j-1} - p_{j-1} q_j = (-1)^{j-1}, \qquad j \ge -1. 
\eeq

Given a finite continued fraction (\ref{ficf}), % for $n\geq 2$, 
written as 
$[a_0; {\bf a}]$, where ${\bf a}=(a_1,a_2,\ldots,a_n)$ is the word of length $n$ 
defining the fractional part, let ${\bf a}^R=(a_n,a_{n-1},\ldots,a_1)$ denote the 
reversed word, and introduce the modified word of length $n+1$ given by 
$$ 
\tilde{{\bf a}}=(a_1,a_2, \ldots, a_{n-1},a_n-1,1), 
$$ 
together with its reversal $\tilde{{\bf a}}^R$. 
Then it is 
convenient to define the following two families of transformations, labelled by a parameter $z$: 
\beq\label{zfam} 
\varphi_z^{(+1)} : \, [a_0; {\bf a}] \mapsto [a_0; {\bf a}, z-1, \tilde{{\bf a}}^R], 
\quad 
\varphi_z^{(-1)} : \, [a_0; {\bf a}] \mapsto [a_0; \tilde{{\bf a}}, z-1, {\bf a}^R]. 
\eeq %where the word $\hat{{\bf a}}$ is given by 
These operators are analogous to the folding maps employed in \cite{folded}; with a slightly different 
notation, the operator $\varphi_z^{(+1)}$ was defined previously in \cite{HoneVar}.  
For words of length zero, the action of these operators is defined by 
\beq\label{zerolength} 
\varphi_{z}^{(+1)} ( [a_0]) = [a_0; z-1,1], 
\quad 
\varphi_{z}^{(-1)} ( [a_0] ) = [a_0-1; 1,z-1]. 
\eeq
(In what follows, we will sometimes omit the subscript $z$, but the implicit dependence on a parameter 
$z$ should be understood.) 

For each $z$, starting from a continued fraction of length 
$\ell ([a_0; {\bf a}]) =n$, 
%final partial quotient has index $n$, 
each of the operators $\varphi_z^{(\pm 1)}$ generically produces a new continued fraction of length 
$\ell\big(\varphi_z^{(\pm 1)}([a_0;{\bf a}])\big)=2n+2$. 
However, if it happens that $z=1$ or $a_n=1$ in (\ref{zfam}), then a zero coefficient will appear 
in the continued fraction obtained by applying one of these operators, and so 
in order to obtain only positive partial quotients 
one must use the concatenation operation 
\beq\label{concat} 
[\ldots, A, 0, B, \ldots] \mapsto [\ldots, A+B, \ldots]
\eeq 
(see e.g.\ Proposition~3 in \cite{folded}), which reduces the length by~$2$. 
Henceforth we will assume that, whenever this occurs, the action of $\varphi_z^{(\pm1)}$ is understood 
as producing the result of concatenation of any zero that appears. 

Our interest in the above transformations is due to 
\begin{lemma}\label{trans} 
\beq\label{pqsum} 
\frac{p_n}{q_n}\pm \frac{(-1)^n}{zq_n^2} = \varphi_z^{(\pm 1)}([a_0;{\bf a}]). 
\eeq 
\end{lemma}

\begin{proof} This is a version of what is referred to as the Folding Lemma in \cite{adamc}, where it is 
attributed to Mend\`es France \cite{MF}. 
The formula for $\varphi_z^{(+ 1)}$ is Lemma~4.1 in \cite{HoneVar}, and is also a corollary of Proposition~2 in \cite{folded}, so we just give details of the proof for $\varphi_z^{(-1)}$. 
Using matrix identities, similarly to  
the proof of 
Proposition~2.1 in \cite{Hone2}, we define 
$$ 
{\bf A}_{a}:=\left(\begin{array}{cc} a & 1 \\ 
1 & 0 \end{array} \right) 
$$
so that 
$$ 
{\bf M}_n := {\bf A}_{a_0} {\bf A}_{a_1}\cdots {\bf A}_{a_n} 
= \left(\bear{cc} p_n & p_{n-1}\\ 
q_n & q_{n-1}
\eear\right), 
$$ 
by (\ref{mat}). 
Then the continued fraction $\varphi_z^{(- 1)}([a_0;{\bf a}]) = [a_0; \tilde{{\bf a}}, z-1, {\bf a}^R]$ 
of length $2n+2$ corresponds to the matrix product 
$$ 
\begin{array}{rcl} 
\tilde{{\bf M}}_{2n+2} & := & 
{\bf A}_{a_0} {\bf A}_{a_1}\cdots {\bf A}_{a_{n-1}} {\bf A}_{a_n-1} {\bf A}_{1} {\bf A}_{z-1} 
{\bf A}_{a_n} {\bf A}_{a_{n-1}}\cdots {\bf A}_{a_{1}} \\ 
& = & {\bf M}_n {\bf A}_{a_n}^{-1} {\bf A}_{a_n-1} {\bf A}_{1} {\bf A}_{z-1} {\bf M}_n^T {\bf A}_{a_0}^{-1} 
\end{array} 
$$ 
which simplifies further to give 
$$
\begin{array}{rcl} 
 \tilde{{\bf M}}_{2n+2}& = & \left(\bear{cc} \tilde{p}_{2n+2} & \tilde{p}_{2n+1}\\ 
 \tilde{q}_{2n+2} & \tilde{q}_{2n+1}
\eear\right) \\
& = & \left(\bear{cc} p_n & p_{n-1}\\ 
q_n & q_{n-1}
\eear\right) \left(\bear{cc} z & 1\\ 
-1 & 0
\eear\right) \left(\bear{cc} q_n & p_{n}-a_0q_{n}\\ 
q_{n-1} & p_{n-1}-a_0q_{n-1}
\eear\right), 
\end{array} 
$$ 
where the entries in the first column of $\tilde{{\bf M}}_{2n+2}$ are 
$$ 
\tilde{p}_{2n+2}=zp_nq_n+p_nq_{n-1}-p_{n-1}q_n, \qquad 
\tilde{q}_{2n+2}=zq_n^2. 
$$ 
Then, from the determinant formula (\ref{detid}) it follows that the final convergent of the finite continued fraction 
$[a_0; \tilde{{\bf a}}, z-1, {\bf a}^R]$ 
is 
$$ 
\frac{\tilde{p}_{2n+2}}{\tilde{q}_{2n+2}}= \frac{zp_nq_n+(-1)^{n-1}}{zq_n^2} = \frac{p_n}{q_n} - \frac{(-1)^n}{zq_n^2}, 
$$ 
as required. 
\end{proof}

We now have the necessary tools to describe the continued fraction expansion of a series 
(\ref{rndsign}), subject to the strong Engel property (\ref{strong}). For $\eps_1=\pm 1$, after subtracting 
the integer part, we arrive at a series of the form 
\beq\label{pqstrong} 
S=\frac{p}{q} +\sum_{n=2}^\infty \frac{\eps_n}{x_n}, \quad q=x_1, \quad \eps_n=\pm 1, 
\eeq 
with $0<p/q<1$; so we proceed to describe the continued fraction expansion for a sum 
of the form (\ref{pqstrong}) with an arbitrary positive rational number $p/q$. 
Due to the property of finite continued fractions that 
\beq
\label{propr}
[a_0; {\bf a}, b,1]=[a_0; {\bf a},b+1], 
\eeq 
we can always write $p/q$ in the form 
\beq\label{pqfrac} 
\frac{p}{q} = [a_0; a_1,\ldots,a_{k}], \qquad a_k>1. 
\eeq 

The main result of this section is the following %broad 
generalization of Theorem~1 in~\cite{folded}.

\begin{theorem}\label{main}
Given a rational number $p/q$ in lowest terms, if it is an integer then write $p/q=a_0=[a_0]$, or otherwise 
let $[a_0; {\bf a}]$ be its continued fraction expansion (\ref{pqfrac}) 
of length $k>0$. Then, subject to the strong Engel property 
(\ref{strong}), the continued fraction expansion of the series (\ref{pqstrong}) is 
given in terms of $[a_0; {\bf a}]$ and the integer parameters $z_j=x_j/x_{j-1}^2>0$ for $j\geq 2$ by 
\beq\label{scf}
S = \prod_{j=3}^\infty \varphi_{z_j}^{(\eps_j)}
\big( \varphi_{z_2}^{(\eps_2(-1)^k)}([a_0; {\bf a}])\big)   
= \cdots  \varphi_{z_4}^{(\eps_4)} \varphi_{z_3}^{(\eps_3)} 
\varphi_{z_2}^{(\eps_2(-1)^k)}([a_0; {\bf a}]) .
\eeq 
Suppose further that $k>0$ and all $z_j>1$, with $a_1>2$ if $k=1$, and $a_1>1$ otherwise. % if $k\geq 2$. 
Then % generic case 
the length $\ell_n = \ell(S_n)$ of the continued fraction for $S_n$, the $n$th partial sum of the series, is 
\beq\label{lensn}
\ell_n = (k+2)2^{n-1}-2. 
\eeq 
\end{theorem} 

\begin{proof} 
The formula (\ref{scf}) for $S$ follows by induction from Lemma~\ref{trans}, using the fact 
that 
$$ 
S_{n+1} = S_n + \frac{\eps_{n+1}}{x_{n+1}} = \frac{p_{\ell_n}}{q_{\ell_n}}+
\frac{\eps_{n+1}}
{z_{n+1}q_{\ell_n}^2}, 
$$ 
so the continued fraction for $S_{n+1}$ is obtained by applying the operator $\varphi_{z_{n+1}}^{(\eps_{n+1})}$ 
to the continued fraction for $S_{n}$, 
except if $n=1$ and $k$ is odd, when one should apply $\varphi_{z_{2}}^{(-\eps_{2})}$ instead; and note that the length $\ell_n$ is even for $n>1$, since the 
operators $\varphi_{z}^{(\pm 1)}$ always produce an even length continued fraction. The exact expression (\ref{lensn}) for 
the lengths in the case of $k>0$ and generic parameters follows immediately from the recurrence 
$\ell_{n+1}=2\ell_n +2$ with the initial value $\ell_1=k$. 
\end{proof} 

\begin{remark} 
Theorem~2.3 in \cite{Hone2} covers the special case $p/q=1$ with $\eps_j=+1$ for all $j$, while 
Theorem~4.2 in \cite{HoneVar} is the case of $k$ even with all $\eps_j=+1$.
\end{remark} 

It is worth briefly commenting on the non-generic cases, when the formula (\ref{lensn}) 
is no longer valid. % the above result. 
If $z_j=1$ for some $j\geq 2$ then 
a zero appears and concatenation reduces the length of the continued fraction by~$2$. 
When $k=0$, starting from $[a_0]=a_0$ we see from (\ref{zerolength}) %have 
%$$ \varphi_{z_2}^{(+1)} ( [a_0]) = [a_0; z-1,1], \quad \varphi_{z_2}^{(-1)} ( [a_0] ) = [a_0; 1,z-1], $$ 
%so the presence of 
that $\varphi_{z_2}^{(+1)} ( [a_0])$ and $\varphi_{z_2}^{(-1)} ( [a_0])$ include   
1 as a partial quotient, which means that in the first case a single application of  
$\varphi^{(+1)}$ or $\varphi^{(-1)}$ produces a zero, while in the second case applying $\varphi^{(+1)}$ or $\varphi^{(-1)}$ moves the 1 
to the end, so that a zero appears at the next step. 

Otherwise, for $k\geq 1$ the initial application of 
$\varphi_{z_2}^{(\pm1)}$ moves the coefficient $a_1$ to the end of the continued fraction, where it remains 
thereafter, so if 
$a_1=1$ then a zero appears at the next step. 
In the particular case $k=1$, applying $\varphi^{(+1)}$ twice sends 
%the steps are $ 
$$ 
\begin{array}{rcl} 
[a_0; a_1] & \mapsto & [a_0;a_1,z_2-1,1,a_1-1] 
\\ 
& \mapsto & 
[a_0;a_1,z_2-1,1,a_1-1,z_3-1,1,a_1-2,1,z_2-1,a_1], 
\end{array}  
$$
and the situation is similar 
when $\varphi^{(+1)}$ is followed by 
$\varphi^{(-1)}$,  
while $\varphi^{(-1)}$ followed by 
$\varphi^{(+1)}$ sends 
%the steps are 
$$ 
\begin{array}{rcl} 
[a_0; a_1]& \mapsto & [a_0;a_1-1,1, z_2-1,a_1] \\ 
& \mapsto & 
[a_0;a_1-1,1,z_2-1,a_1,z_3-1,1,a_1-1,z_2-1,1,a_1-1],
\end{array}  
$$ 
so that $a_1-2$ appears at the next step, and similarly when $\varphi^{(-1)}$ is applied twice. So 
the case $a_1=2$ is also degenerate when $k=1$. 

The latter considerations allow us to state the following corollary of Theorem~\ref{main}, 
which will be relevant to the series of L\"uroth type treated in the next section. 

\begin{corollary}\label{lencor} 
For a strong Engel series with signs, that is 
\beq\label{sstreng} 
\frac{1}{x_1}+\sum_{j=2}^\infty \frac{\eps_j}{x_j}, \qquad \eps_j=\pm 1,
\eeq 
with $z_1=x_1>1$, $z_j = x_{j}/x_{j-1}^2>1$ for all $j\geq 2$, 
the $n$th partial sum $S_n$ has length 
\beq \label{genericl} 
\ell_n = 3\cdot 2^{n-1} -2, \qquad n\geq 1,
\eeq 
in the generic case that $x_1>2$. In the special case $x_1=2$, for a strong Engel series with 
$\eps_j=1$ for all $j$,  
the formula should be modified to 
\beq \label{specengel} 
\ell_n = 5\cdot 2^{n-2}, \qquad n\geq 3, 
\eeq 
while for a strong Pierce series with $\eps_j=(-1)^{j-1}$, the formula becomes 
\beq \label{specpierce} 
\ell_n = 5\cdot 2^{n-2}-2, \qquad n\geq 3. 
\eeq 
\end{corollary} 

\begin{proof}[Proof of Corollary]%%~\ref{lencor}]
This is the case $k=1$ of Theorem~\ref{main} with $a_0=0$ and $a_1=x_1=z_1$, so the generic 
length formula (\ref{lensn}) yields (\ref{genericl}) immediately. In the degenerate case $x_1=a_1=2$, 
for a strong Engel series one must start by applying  
$\varphi^{(-1)}$ followed by 
$\varphi^{(+1)}$, which sends 
%the steps are 
$$
[0; 2]\mapsto 
[0;1,1, z_2-1,2]\mapsto 
[0;1,1,z_2-1,2,z_3-1,1,1,z_2-1,1,1], 
$$ 
while at the next stage $\varphi^{(+1)}$ together with (\ref{concat}) produces 
\small 
$$
[0;1,1,z_2-1,2,z_3-1,1,1,z_2-1,1,1,z_4-1,2,z_2-1,1,1,z_3-1,2,z_2-1,1,1]. 
$$ 
\normalsize
Then, because 
%and as 
this and all subsequent continued fractions begin %with 
$[0;1,1,\ldots]$, each successive application of $\varphi^{(+1)}$
requires concatenation, so the recurrence for the lengths is modified to $\ell_{n+1}=2\ell_n$ for $n\geq 3$, which  
gives the formula (\ref{specengel}). For a strong Pierce series one should repeatedly apply 
$\varphi^{(+1)}$ followed by 
$\varphi^{(-1)}$, so that with $x_1=2$ the sequence begins 
$$
[0; 2]\mapsto 
[0;2, z_2-1,1,1]\mapsto 
[0;2,z_2-1,2,z_3-1,1,1,z_2-1,2], 
$$ 
where we used (\ref{concat}) in the second step, 
and thereafter there is always a $2$ at the beginning and end of each continued fraction, so 
no further concatenation is required and we just have the generic recursion 
$\ell_{n+1}=2\ell_n +2$ for $n\geq 3$, giving (\ref{specpierce}). 
\end{proof} 

\begin{remark}
Note that in the strong Engel case with $x_1=2$ 
the %sequence of 
continued fractions from $n=3$ onwards all end with $[\ldots,1,1]$, so 
we could instead use (\ref{propr}) to make them end with $[\ldots,2]$ and write a reduced length 
formula 
$$
\ell_n = 5\cdot 2^{n-2}-1
$$ 
in that case. However, the non-reduced continued fractions 
maintain the property of having even length at each stage, which we prefer to keep. 
Similarly, in the strong Pierce case we instead take the reduced continued fractions ending in $2$ for the same reason.
\end{remark} 

%----------------
\section{L\"{u}roth series and generalizations}
\label{luroths}
%----------------

It was shown by L\"uroth that every real number in the interval $(0,1)$ admits an expansion of the form 
(\ref{luroth}) for a certain sequence of integers $u_j\geq 2$ \cite{lurothoriginal}. As well 
as being used for Diophantine approximation of real numbers \cite{CWZ, DaKr}, 
L\"uroth series have also been employed in the context of rational function 
approximations of power series \cite{KnKn1, KnKn2}. 

In \cite{kkk}, Kalpazidou et al.\ introduced two different alternating analogues of L\"uroth series, each of 
which provides a unique representation of a real number. For a number in $(0,1)$, the first 
type of alternating L\"uroth expansion defined in \cite{kkk} takes the form 
\beq\label{altluroth} 
 \frac{1}{u_1} +\sum_{j=2}^{\infty} \frac{(-1)^{j-1}}{u_{1}(u_{1}+1) \cdots 
  u_{j-1}(u_{j-1}+1) u_{j}}, 
\eeq 
for a sequence of integers $u_j\geq 1$. 

It is clear from the form of (\ref{luroth}) that if we set $x_1=u_1$ and 
$$ 
x_j =u_j \prod_{k=1}^{j-1} u_k(u_k-1), \qquad j\geq 2, 
$$ 
then $x_j \mid x_{j+1}$, so a L\"uroth series is a particular example of an Engel series, 
and similarly an alternating L\"uroth series (\ref{altluroth}) is a particular type of Pierce series. 

In order to consider these two examples together it is convenient to start with a more general family of series, of the form 
\beq\label{sfam}
  S' = \frac{1}{u_1} +\sum_{j=2}^{\infty} \frac{\eps_{j}}{u_{1}v_1 \cdots 
  u_{j-1}v_{j-1} u_{j}}, \quad \eps_j=\pm 1, 
\eeq 
for sequences of integers $u_j,v_j\in\Zp$. 
If we now take 
\beq\label{xudef} 
x_1=u_1, \qquad 
x_j =u_j \prod_{k=1}^{j-1} u_kv_k, \qquad j\geq 2, 
\eeq 
then it turns out we can further obtain the strong Engel property (\ref{strong}) for the sequence 
$(x_n)$ whenever $(u_n)$ satisfies certain recurrence relations of second order, analogous to~(\ref{2nd}). 

\begin{proposition}\label{recprop} 
Suppose that the sequence $(u_n)$ satisfies 
either the recurrence 
\beq\label{recjlv}
u_{n+2}u_n = \al_{n} u_{n+1}^3 v_{n+1}, \qquad n\geq 1, 
\eeq 
where $u_2=mu_1^2v_1$, 
or 
\beq\label{recanwh} 
u_{n+2} = \al_{n} u_{n+1}^2v_n, \qquad n\geq 1, 
\eeq 
where $u_2=mu_1$, and 
in each case $(\al_n)$ is an arbitrary sequence of positive integers, with
$u_1,m\in\Zp$ arbitrary. 
Then 
the associated sequence $(x_n)$ defined by (\ref{xudef}) has the strong Engel property, that is 
$z_j =x_{j}/x_{j-1}^2 \in \Z$ holds for all $j\geq 2$. 
\end{proposition} 

\begin{proof} 
For the proof 
it is convenient to write the various relations between $u_n,v_n,\al_n$ and $x_n$ in terms 
of the exponentiated shift operator, as in (\ref{eshift}). From (\ref{xudef}) we have 
$$ 
\theta [ x_n] = u_nv_n \theta [u_n], \qquad n\geq 1, 
$$ 
which implies 
$$
\theta^2 [ x_n] = \theta [u_n]\theta[v_n] \theta^2 [u_n]
$$ 
also holds for $n\geq 1$.
Then, by rewriting the first recurrence (\ref{recjlv}) as 
$$ 
\theta^2[u_n]=\al_n u_{n+1}v_{n+1}, 
$$ 
we calculate 
$$
z_2=\frac{x_2}{x_1^2}=\frac{u_2v_1}{u_1} =mu_1v_1^2\in\Zp, 
$$ 
while for $n\geq 1$ we have 
$$ 
z_{n+2}=\al_nu_{n+1}v_{n+1}^2 \rho_n, 
$$ 
with 
$$ 
\rho_n := \frac{\theta[u_n]}{x_nv_n}, \qquad n\geq 1. 
$$ 
The latter definition gives $\rho_1=u_2/(u_1^2v_1)=m$ and 
\beq \label{rhorec} 
\theta [\rho_n]=\frac{\theta^2[u_n]}{\theta[x_n]\theta[v_n]}=\al_n, 
\eeq 
which just says that $\rho_{n+1}=\al_n\rho_n$, so by induction we have 
$\rho_n\in\Zp$ for all $n\geq 1$, and this implies $z_j\in\Zp$ for all $j\geq 2$, as 
required. 
Similarly, we rewrite the second recurrence (\ref{recanwh}) as 
$$ 
\theta^2[u_n]=\al_n u_{n}v_{n}, 
$$ 
then calculate 
$$
z_2=\frac{u_2v_1}{u_1} =mv_1\in\Zp, 
$$ 
and for $n\geq 1$ we find 
$$ 
z_{n+2}=\al_n v_{n+1} \rho_n, 
$$ 
where in this case we instead take the definition 
$$ 
\rho_n := \frac{u_n\theta[u_n]}{x_n}, \qquad n\geq 1. 
$$ 
The latter definition gives $\rho_1=m$ once again, and also 
$$ 
\theta [\rho_n]=\frac{\theta[u_n]\theta^2[u_n]}{\theta[x_n]}=\al_n, 
$$ 
which is the same final result for $\theta [\rho_n]$ as in (\ref{rhorec}), 
so the conclusion is the same. 
\end{proof} 

\begin{example}\label{lur1} 
Taking $\al_n=1$, $v_n=u_n-1$ for all $n\geq 1$, the recurrence (\ref{recjlv}) becomes 
\beq 
\label{jlv1} 
u_{n+2}u_n=u_{n+1}^3 (u_{n+1}-1), 
\eeq 
and setting $u_1=3$, $m=1$ gives $u_2=u_1^2(u_1-1)=18$, so the sequence $(u_n)$ begins with 
\beq\label{lurseq} 
3,18,33048, 66266659938624768,\ldots . 
\eeq 
We have $x_1=3$, $x_{n+1}=x_nu_{n+1}(u_n-1)$ for $n\geq 1$. Hence 
$$ 
3, 108, 60676128, 132875521042766180738219532288,\ldots 
$$
is the beginning of the sequence $(x_n)$. 
Then we find 
$
z_2={108}/{3^2} = 12$, 
$ 
z_3={60676128}/{108^2} = 5202$, 
$
z_4={132875521042766180738219532288}/{60676128^2} = 36091859899032$, 
and so on. 
\end{example} 

\begin{example}\label{altlur2} 
Taking $\al_n=1$, $v_n=u_n+1$ for all $n\geq 1$, the recurrence (\ref{recanwh}) becomes 
\beq 
\label{anwh1} 
u_{n+2}=u_{n+1}^2 (u_{n}+1), 
\eeq 
and taking $u_1=2$, $m=1$ gives $u_2=u_1=2$, so the sequence $(u_n)$ begins with 
\beq\label{altlurseq}
2,2,12,432,2426112,\ldots .
\eeq 
We have $x_1=2$, $x_{n+1}=x_nu_{n+1}(u_n+1)$ for $n\geq 1$. Hence 
$$ 
2,12,432,2426112,2548646416023552,\ldots 
$$
is the beginning of the sequence $(x_n)$, and it is not hard to 
show that in fact $x_{n}=u_{n+1}$ holds for all $n\geq 1$, with this particular choice of 
initial values for~(\ref{anwh1}). 
Then we find 
$
z_2={12}/{2^2}=3$, 
$ 
z_3=432/{12^2}=3$, 
$
z_4={2426112}/{432^2}=13$, 
and in general $z_n=u_{n-1}+1$ for all $n\geq 2$. 
\end{example} 

We can now combine the results in section~\ref{sec:ecf} with Proposition~\ref{recprop} to describe 
the continued fraction expansion of certain L\"uroth series with the strong 
Engel property, not just of the form (\ref{luroth}) but 
also with arbitrary signs inserted. 

\begin{theorem}\label{lurthm}
Suppose that the sequence $(u_n)$ satisfies 
either the recurrence 
\beq\label{recjlvlur}
u_{n+2}u_n = \al_{n} u_{n+1}^3 (u_{n+1}-1), \qquad n\geq 1, 
\eeq 
where $u_2=mu_1^2(u_1-1)$, 
or 
\beq\label{recanwhlur} 
u_{n+2} = \al_{n} u_{n+1}^2(u_n-1), \qquad n\geq 1, 
\eeq 
where $u_2=mu_1$, and 
in each case $(\al_n)$ is an arbitrary sequence of positive integers, with
$u_1\in\Z_{>1}$, $m\in\Z_{>0}$ arbitrary. 
Then the continued fraction expansion of the sum 
\beq\label{sluroth} 
  S = \frac{1}{u_1} +\sum_{j=2}^{\infty} \frac{\eps_j}{u_{1}(u_{1}-1) \cdots 
  u_{j-1}(u_{j-1}-1) u_{j}}, \qquad \eps_j=\pm 1,
\eeq 
is given by 
\beq\label{lurcf}
  S = \prod_{j=3}^\infty \varphi_{z_j}^{(\eps_j)}
  \big( \varphi_{z_2}^{(-\eps_2)}([0; u_1])\big)  
  = \cdots \varphi_{z_4}^{(\eps_4)} \varphi_{z_3}^{(\eps_3)} 
  \varphi_{z_2}^{(-\eps_2)}([0; u_1]) ,
\eeq 
where either $z_{n+1}=u_{n}(u_{n}-1)^2\rho_{n}$ in the case that (\ref{recjlvlur}) holds, 
or $z_{n+1}=(u_{n}-1)\rho_{n}$ when (\ref{recanwhlur}) holds, with 
\beq\label{rhoform} 
\rho_n = m \prod_{k=1}^{n-1}\al_k
\eeq 
in both cases, for all $n\geq 1$. 
\end{theorem} 

\begin{proof} 
This follows by combining Theorem~\ref{main} for $k=1$ with Proposition~\ref{recprop}, 
in the particular case that $v_n=u_n-1$ for all $n$, and noting that 
from (\ref{rhorec}) 
we have $\rho_{n+1}=\al_n \rho_n$ for $n\geq 1$ with the initial value $\rho_1=m$, 
which yields both (\ref{rhoform}) and the appropriate formula for $z_{n+1}$ according to 
whether (\ref{recjlvlur}) or (\ref{recanwhlur}) holds.
\end{proof}

\begin{example}\label{lur1cont}
As a continuation of Example \ref{lur1}, 
it follows that the number $S\approx 0.34259260907$ whose L\"uroth series %expansion 
is defined 
by the sequence (\ref{lurseq}), that is 
$$ 
S=\frac{1}{3}+\frac{1}{108}+\frac{1}{60676128}+\frac{1}{132875521042766180738219532288}+\cdots, 
$$ 
has continued fraction expansion 
\small
$$
[0;2,1,11,3,5201,1,2,11,1,2,36091859899031,1,1,1,11,2,1,5201,3,11,1,2,\ldots].
$$
\normalsize
The infinite continued fraction is obtained by folding the sequence of finite continued fractions for 
the $n$th truncation of the series, that is 
$$
[0;3]\mapsto [0;2,1,11,3]\mapsto [0;2,1,11,3,5201,1,2,11,1,2]\mapsto \cdots,
$$ 
where the lengths are given by the formula (\ref{genericl}), and this pattern of lengths 
remains the same if arbitrary signs are inserted in $S$. 
\end{example} 

It is straightforward to state the analogue of Theorem~\ref{lurthm} for the case 
of an alternating L\"uroth series (\ref{altluroth}), also with the inclusion of arbitrary 
signs. The proof is essentially the same so is omitted. 

\begin{theorem}\label{altlurthm} 
Suppose that the sequence $(u_n)$ satisfies 
either the recurrence 
\beq\label{recjlvaltlur}
u_{n+2}u_n = \al_{n} u_{n+1}^3 (u_{n+1}+1), \qquad n\geq 1, 
\eeq 
where $u_2=mu_1^2(u_1+1)$, 
or 
\beq\label{recanwhaltlur} 
u_{n+2} = \al_{n} u_{n+1}^2(u_n+1), \qquad n\geq 1, 
\eeq 
where $u_2=mu_1$, and 
in each case $(\al_n)$ is an arbitrary sequence of positive integers, with 
$u_1\in\Z_{>1}$, $m\in\Z_{>0}$ arbitrary. 
Then 
the continued fraction expansion of the sum 
\beq\label{saltluroth} 
  S' = \frac{1}{u_1} + \sum_{j=2}^{\infty} \frac{\eps_j}{u_{1}(u_{1}+1) \cdots 
  u_{j-1}(u_{j-1}+1) u_{j}}, \qquad \eps_j=\pm 1,
\eeq 
is given by the same formula as for $S$ in (\ref{lurcf}), 
but with $z_{n+1}=u_{n}(u_{n}+1)^2\rho_{n}$ when (\ref{recjlvaltlur}) holds, 
or $z_{n+1}=(u_{n}+1)\rho_{n}$ when (\ref{recanwhaltlur}) holds, with 
$\rho_n = m \prod_{k=1}^{n-1}\al_k$ in both cases, for all $n\geq 1$. 
\end{theorem}

\begin{example}\label{altlur2cont}
As a continuation of Example \ref{altlur2}, 
it follows that the number $S'\approx 0.418981069299$ whose alternating L\"uroth expansion is defined 
by the sequence (\ref{altlurseq}), that is 
$$ 
S'=\frac{1}{2}-\frac{1}{12} +\frac{1}{432}-\frac{1}{2426112}+\frac{1}{2548646416023552}-\cdots, 
$$ 
has continued fraction expansion 
\small
$$
[0;2,2,1,1,2,2,2,1,1,12,2,2,2,2,1,1,2,2,432,1,1,2,1,1,2,2,2,2,12,1,1,\ldots].
$$
\normalsize 
The infinite continued fraction is obtained by folding the sequence of finite continued fractions 
for the $n$th truncation of the series $S'$, that is 
$$ 
\begin{array}{rcl} 
[0;2]& \mapsto & [0;2,2,1,1] \\ 
& \mapsto & [0;2,2,1,1,2,2,2,2] \\ 
& \mapsto & 
[0;2,2,1,1,2,2,2,1,1,12,2,2,2,2,1,1,2,2]%\cdots, 
\end{array} 
$$ 
etc., and since $k=1$ and $a_1=x_1=2$ this is a non-generic case, with the lengths being given 
by the formula (\ref{specpierce}) for $n\geq 3$. 
\end{example} 

The %preceding 
result of Proposition~\ref{recprop} 
requires the sequence $(u_n)$ to satisfy one of the 
recurrences (\ref{recjlv}) or (\ref{recanwh}), which depend on how the 
sequence $(v_n)$ is specified, for instance, imposing $v_n=u_n-1$ 
%must be imposed in the case of 
for a L\"uroth series (\ref{luroth}), or %while 
$v_n=u_n+1$ %is required 
for an alternating L\"uroth series (\ref{altluroth}), as above. 
However, there is another way to obtain the strong Engel property, by imposing 
independent conditions on the sequences $(u_n)$ and $(v_n)$. 

\begin{proposition} 
Suppose that the sequences $(u_n)$ and $(v_n)$ satisfy 
\beq\label{recuv}
u_n = \bet_{n} \prod_{k=1}^{n-1} u_{k}, 
\quad v_n = \gam_{n} \prod_{k=1}^{n-1} v_{k}, \qquad n\geq 2, 
\eeq 
where 
$(\bet_n)$ and $(\gam_n)$ are arbitrary sequences of positive integers, with 
arbitrary $u_1,v_1\in\Z_{>0}$. 
Then the associated sequence $(x_n)$ defined by (\ref{xudef}) has the strong Engel property, that is 
$z_j =x_{j}/x_{j-1}^2 \in \Z$ holds for all $j\geq 2$. 
\end{proposition} 

\begin{proof} 
We have 
$$
z_2=\frac{x_2}{x_1^2}=\frac{u_2u_1v_1}{u_1^2}=\frac{u_2v_1}{u_1}=\bet_2v_1, 
$$
while 
for $j\geq 2$ we see that 
$$ %\begin{array}{rcl} 
z_{j+1} = \frac{u_{j+1}\prod_{k=1}^ju_kv_k}{u_j^2\left(\prod_{k=1}^{j-1}u_kv_k\right)^2} %\\ 
= \frac{u_{j+1}v_j}{u_j\prod_{k=1}^{j-1} u_k v_k } %\\ 
= \beta_{j+1}\gam_j, 
%\end{array} 
$$ 
and the result follows. 
\end{proof} 

\begin{example}\label{zjisj}
Upon setting $\bet_n=n$, $\gam_n=1$ for all $n\geq 2$ and $u_1=v_1=1$, we have 
$v_n=1$ for all $n$, and we find 
$$ 
u_n=n\prod_{k=1}^{n-2}(n-k)^{2^{k-1}}, \qquad 
x_n = \prod_{k=0}^{n-2}(n-k)^{2^{k}},
$$ 
where $x_n=\prod_{k=1}^nu_k$ in this case, 
which implies that $z_n=\bet_n=n$ for $n\geq 2$. So the sequence $(u_n)$ begins 
with $1,2,6,48,2880,9953280,\ldots$, and 
$(x_n)$ begins with $1,2,12,576,1658880,16511297126400,\ldots$. The 
alternating sum $\sum_{j\geq 1}(-1)^{j-1}/x_j$ is the strong Pierce series 
\beq\label{expo2}
S'=1-\frac{1}{2}+\frac{1}{12}-\frac{1}{576}+\frac{1}{1658880}-\frac{1}{16511297126400}+\cdots, 
\eeq 
and the continued fraction expansion of $S'\approx 0.5815978250$ is 
\small 
$$ 
[0;1,1,2,1,1,3,2,2,1,1,4,2,2,2,3,1,1,2,2,5,1,1,2,1,1,3,2,2,2,4,1,1,2,2,\ldots].
$$
\normalsize 
The corresponding sequence of foldings of finite continued fractions begins 
$$ 
\begin{array}{rcl}
[1] & \mapsto & [0;1,1]=[0;2] \\ 
& \mapsto & [0;1,1,2,2] \\ 
& \mapsto & [0;1,1,2,1,1,3,2,2,1,1] \\ 
& \mapsto & [0;1,1,2,1,1,3,2,2,1,1,4,2,2,2,3,1,1,2,1,1], 
\end{array} 
$$ 
and so on, viewed as corresponding to $k=0$ in Theorem~\ref{main}, or to $k=1$ if we combine 
the first two terms so that $S'=1/x_1'+\sum_{j\geq 2}(-1)^j/x_j'=1/2 + 1/12-1/576+1/1658880-\cdots$, 
with $x_1'=2$, $x_j'=x_{j+1}$ for $j\geq 2$, and then the sequence 
of lengths is given by the formula (\ref{specengel}). 
\end{example}

%----------------
\section{Irrationality exponents and transcendence} 
%----------------

In this final section we compute the irrationality exponents of certain families of transcendental 
numbers defined by series of L\"uroth/alternating L\"uroth type, with arbitrary signs, that have the strong Engel 
property, before concluding with a conjecture concerning the whole family of series in Theorem~\ref{main}. 

Recall that the irrationality exponent $\mu(\xi)$ of a real number $\xi$ is defined to be the 
supremum of the set of real numbers $\mu$ such that there are infinitely many 
rational approximations $p/q$ satisfying the inequality 
$$
|\xi - p/q| < 1/q^\mu.
$$ 
For an irrational number, $\mu(\xi)\geq 2$, and 
the irrationality exponent is given in terms of $q_n$, the denominators of the convergents 
of the continued fraction expansion of $\xi$, by the formula 
\beq\label{limsup} 
\mu(\xi) = 1+ \limsup\limits_{n\to\infty} \frac{\log q_{n+1}}{\log q_n}.
%=2+\limsup\limits_{n\to\infty} \frac{\log a_{n+1}}{\log q_n} 
\eeq 
If $\mu(\xi)>2$ then $\xi$ is transcendental, by Roth's theorem \cite{Roth}. 

\begin{theorem}\label{irratexp} 
Suppose that a number $\xi\in\R_{>0}$ is defined by either a series of the L\"uroth type 
(\ref{sluroth}) with arbitrary signs, subject to a recurrence of the form (\ref{recjlvlur}), 
or of the alternating L\"uroth type 
(\ref{saltluroth}) with arbitrary signs, subject to a recurrence of the form (\ref{recjlvaltlur}), 
where in each case $\al_n$ is given by 
\beq \label{alform} 
\al_n= \left \lceil{\exp (C\nu^n)}P(u_n,u_{n+1})\right \rceil, 
\text{ with } P(X,Y) = \sum_{i=0}^M \sum_{j=0}^N c_{ij} X^{r_i}Y^{s_j},
\eeq 
%with $$ P(X,Y)=\sum_{i=0}^M \sum_{j=0}^N c_{ij} X^{r_i}Y^{s_j},$$ 
for non-negative integers $M,N$, positive real numbers $C,c_{ij},\nu$, and non-negative real exponents
$r=r_M>r_{M-1}>\cdots>r_0\geq 0$, $s=s_N>s_{N-1}>\cdots>s_0\geq 0$. Then 
$\xi$ is transcendental with irrationality exponent 
\beq\label{muf1}
\mu(\xi)=\max \left(\nu, \frac{1}{2}\Big(s+4+\sqrt{(s+4)^2+4(r-1)}\Big)\right). 
\eeq 
Similarly, if $\xi$ is defined by one of the series (\ref{sluroth}) or 
(\ref{saltluroth}), with arbitrary signs, subject to a recurrence of the form 
(\ref{recanwhlur}) or (\ref{recanwhaltlur}), respectively, with $\al_n$ as in 
(\ref{alform}), then it is transcendental 
with irrationality exponent 
\beq\label{muf2}
\mu(\xi) = \max \left(\nu, \frac{1}{2}\Big(s+2+\sqrt{(s+2)^2+4(r+1)}\Big)\right). 
\eeq 
\end{theorem} 

\begin{proof}
For the sake of simplicity, we assume that the series (\ref{sluroth}) or 
(\ref{saltluroth}) being considered is generic, in the sense described in Corollary~\ref{lencor}, 
which means imposing the requirement $u_1\geq 3$, but if this is not the case 
then the same method of proof applies 
with only minor modifications. Clearly $(u_n)$ is an increasing sequence of 
positive integers. 
Upon setting $\La_n=\log u_n$ and taking logarithms in either  
(\ref{recjlvlur}) or (\ref{recjlvaltlur}), subject to (\ref{alform}), we find %that 
\beq \label{linrec}
\La_{n+2}-(s+4)\La_{n+1}+(1-r)\La_n=\Delta_n, \qquad \Delta_n=C\nu^n +o(1). 
\eeq 
By applying the method of Aho and Sloane \cite{ahosloane}, 
the inhomogeneous linear equation (\ref{linrec}) can be solved ``explicitly'' to yield 
\beq\label{Laform} 
\La_n = A\la^n +B\bar{\la}^n + F_n, \qquad F_n = C'\nu^n\big(1 +o(1)\big), 
\eeq 
for certain constants $A,B,C'$, where 
\beq\label{laval1} 
\la=\frac{1}{2}\Big(s+4+\sqrt{(s+4)^2+4(r-1)}\Big)\geq 2+\sqrt{3}, 
\eeq 
and $\bar{\la}$ is the conjugate root of the characteristic quadratic for (\ref{linrec}). 
More details of the precise form of $A,B$ and $F_n$ can be found in \cite{Hone1} (see also \cite{HoneVar}), 
but are not needed here; the formula (\ref{Laform}) is not really 
an explicit solution, because $F_n$ and $A,B$ depend implicitly on the sequence $(u_n)$. Then taking
% if we take 
$$ 
\mu = \max (\nu,\la), 
$$ 
we see that 
\beq\label{lasy} 
\Lambda_n =D\mu^n\big(1+o(1)\big), \qquad D>0, 
\eeq 
where $D$ is either $C'$ or $A$ depending on which of $\nu$ or $\la$ is the greater. 

For what follows, an estimate of the growth of the sequence $(z_n)$ is also required. 
From Theorems \ref{lurthm} and \ref{altlurthm}, 
using (\ref{rhoform}), we have 
$$ 
\bear{rcl} 
\log z_n & = & \log u_{n-1}+2\log(u_{n-1}\mp 1)+\log\rho_{n-1}\\ 
& =& 3\La_{n-1}+\log m+\sum_{k=1}^{n-2}\log\al_k+o(1). 
\eear
$$ 
Thus we see from (\ref{alform}) and (\ref{lasy}) that 
\beq\label{zasy} 
\log z_n = D'\mu^n\big(1+o(1)\big), \qquad D'>0,
\eeq 
where the precise form of $D'$ is unimportant. 

In order to evaluate the limit in (\ref{limsup}), we now consider the three-term recurrence relation for $q_n$ 
encoded in (\ref{mat}), which is $q_{n+1}=a_{n+1}q_n +q_{n-1}$, so
$$ 
L_{n+1}-L_n = \log a_{n+1}+\log\left(1+\frac{q_{n-1}}{a_{n+1}q_n}\right), 
$$ 
where we set $L_n=\log q_n$. 
Performing the telescopic sum of the latter identity, with the initial value 
$L_0 = \log q_0 =0$, and noting that the last term on the right is at most $\log 2$, since $(q_n)$ is an increasing sequence of positive integers and $a_n\geq 1$, 
we obtain 
\beq\label{ineq} 
L_n=\sum_{j=1}^n\log a_{j}+\delta_n, \qquad 0<\delta_n<n \log 2. 
\eeq 

From the discussion before Corollary~\ref{lencor}, it is clear that the only possible values of 
the coefficients appearing in the folded continued fraction (\ref{lurcf}) (or its counterpart as 
described in Theorem~\ref{altlurthm}) are $1$, $u_1$, $u_1-1$, $u_1-2$ and $z_j-1$ for $j\geq 2$.
In an initial block of length $\ell_n=3\cdot 2^{n-1}-2$, as in (\ref{genericl}), the coefficient 
$z_n-1$ appears once, $z_{n-1}-1$ appears twice, and in general $z_j-1$ appears $2^{n-j}$ times. 
This accounts for $2^{n-1}-1$ coefficients out of $\ell_n$, while $\delta_{\ell_n}$ and the 
sum of logarithms of the remaining coefficients are both $O(2^n)$, so 
from (\ref{ineq}) we have 
$$ 
\bear{rcl} 
L_{\ell_n}& = & \sum_{j=2}^n 2^{n-j} \log (z_j-1)+O(2^n) \\[4pt] 
& = & 2^n D' \sum_{j=2}^n (\mu/2)^j \big(1+o(1)\big) + O(2^n) \\[4pt] 
& = & D' (1-2/\mu)^{-1} \mu^n \big(1+o(1)\big),
\eear 
$$ 
since $\mu>2$ by (\ref{laval1}). 
Now if $\eps_{n+1}=+1$ then folding requires an application of $\varphi^{(+1)}_{z_{n+1}}$, 
which gives $a_{\ell_{n}+1}=z_{n+1}-1$ and so 
$$ 
L_{\ell_{n}+1}=L_{\ell_n}+\log(z_{n+1}-1)+\delta_{\ell_n+1}-\delta_{\ell_n}=L_{\ell_n}+D'\mu^{n+1}\big(1+o(1)\big), 
$$ 
which gives 
\beq\label{biggest} 
\frac{L_{\ell_{n}+1}}{L_{\ell_n}}=1+\frac{\mu}{(1-2/\mu)^{-1}}+o(1)=\mu-1+o(1). 
\eeq 
Otherwise, if 
$\eps_{n+1}=-1$ then an application of $\varphi^{(-1)}_{z_{n+1}}$ 
gives $a_{\ell_{n}+1}=1$, so 
${L_{\ell_{n}+1}}/{L_{\ell_n}}=1+o(1)$, but $a_{\ell_n+2}=z_{n+1}-1$, and so 
instead ${L_{\ell_{n}+2}}/{L_{\ell_n+1}}=\mu-1+o(1)$. 
%using the fact that
Then, by considering the sequence of coefficients until the next folding happens 
at length $\ell_{n+1}$, we may write 
$$ 
L_{\ell_n+j}=D'\mu^n(\mu +\tilde{\Delta}_{n,j})\big(1+o(1)\big)
$$ 
for $j\geq 1$ when $\eps_{n+1}=+1$, or for $j\geq 2$ when $\eps_{n+1}=-1$, 
where $0<\tilde{\Delta}_{n,j}=O(1)$ increases by an amount $\mu^{k-n}\leq 1$ each time 
the coefficient $a_{\ell_n+j}$ is equal to $z_k-1$, and otherwise remains the same. So 
there is an initial step where 
${L_{\ell_{n}+j+1}}/{L_{\ell_n+j}}=\mu-1+o(1)$, for $j=0$ or $1$ depending on whether 
$\eps_{n+1}=\pm1$, and at all subsequent steps until 
the next folding we have 
$L_{\ell_{n}+j+1}-L_{\ell_n+j}=D'\mu^n\big(\tilde{\Delta}_{n,j+1}-\tilde{\Delta}_{n,j}+o(1)\big)$ where 
$\tilde{\Delta}_{n,j+1}-\tilde{\Delta}_{n,j}\leq \mu^{k-n}$ 
for $2\leq k\leq n$, so $\tilde{\Delta}_{n,j+1}-\tilde{\Delta}_{n,j}\leq 1$. Hence, for these subsequent steps, 
$$ 
  \frac{L_{\ell_{n}+j+1}}{L_{\ell_n+j}}
  = 1 + \frac{\tilde{\Delta}_{n,j+1} - \tilde{\Delta}_{n,j} + o(1)}
  {(\mu +\tilde{\Delta}_{n,j})\big(1+o(1)\big)}, 
$$
which in the limit %%superior %of these subsequent ratios 
is at most $1+\mu^{-1}$, until the next 
folding happens and there is a term with limit %% superior of 
$\mu-1$, obtained 
from the ratio of terms like (\ref{biggest}). Now $1+\mu^{-1}\leq \mu-1$ for 
$\mu\geq 1+\sqrt{2}$, which holds by (\ref{laval1}). Thus from (\ref{limsup}) we find 
$$ 
\mu(\xi) = 1+\limsup\limits_{n\to\infty} \frac{L_{n+1}}{L_n}=1+\mu-1=\max(\nu,\la), 
$$ 
as required. 

For the second part of the theorem, where $(u_n)$ is subject to (\ref{recanwhlur}) or (\ref{recanwhaltlur}), 
for a series of L\"uroth/alternating L\"uroth type with signs, as appropriate, then 
(\ref{linrec}) is modified to 
$$ 
\La_{n+2}-(s+2)\La_{n+1}-(1+r)\La_n = \Delta_n, \qquad \Delta_n=C\nu^n +o(1), 
$$ 
and the largest characteristic root is 
$$
\la=\frac{1}{2}\Big(s+2+\sqrt{(s+2)^2+4(r+1)}\Big)\geq 1+\sqrt{2}, 
$$
so $\mu=\max(\nu,\la)\geq 1+\sqrt{2}$ still holds, and the rest of the 
proof is the same.
\end{proof} 

\begin{remark} 
A suitable modification of the preceding argument 
should show that the number (\ref{expo2}) defined in Example \ref{zjisj} has irrationality exponent~$2$.
\end{remark}

As is well known, the set of irrational numbers with irrationality exponent greater than $2$ has measure zero. 
If $\mu(\xi)=2$ then there is no simple criterion to decide whether $\xi$ is transcendental or not. 
Nevertheless, we have reason to expect that none of the $\xi$ defined by strong Engel series with signs 
are algebraic. 

\begin{conjecture}\label{transc} 
All of the real numbers $\xi$ defined by a series of the form (\ref{scf}), for arbitrary 
$p/q\in\Q$ and positive integer parameters $z_2,z_3,\ldots$, are transcendental. 
\end{conjecture} 

To explain why the above conjecture is plausible, we consider the case considered in \cite{Hone2}, 
that is $p/q=1$ with all $\eps_j=+1$, when the strong Engel series for $\xi$ has the form 
\beq\label{zser}
S=1+\sum_{j=2}^\infty \frac{1}{z_2^{2^{j-2}}z_3^{2^{j-3}}\cdots z_j} 
\eeq 
(it is necessary to assume that at least one $z_j>1$ to ensure convergence). 
Suppose that we replace the first $n$ of the parameters by variables, so 
$z_{j+1}=\zeta_{j}^{-1}$ for $j=1,\ldots n$, and regard all the other $z_j$ as fixed. 
Then (\ref{zser}) becomes a 
power series 
\beq\label{power} 
S(\zeta_1,\ldots,\zeta_n) = 1+
\sum_{j=1}^\infty c_j \prod_{i=1}^{\min(j,n)}\zeta_i^{2^{j-i}},
\eeq  
for suitable coefficients $c_j$ defined in terms of $z_{n+2},z_{n+3},\ldots$, with 
$c_j=1$ for $1\leq j\leq n$. Then in principle, the series (\ref{power}) should be amenable to the techniques 
of Loxton and van der Poorten \cite{loxtonvdp}, who proved that, subject to some 
recursive systems of functional equations being 
satisfied, certain power series in several variables, with algebraic coefficients, take 
only transcendental values at algebraic points. 

The result of \cite{loxtonvdp} is a very broad 
generalization of a result of Mahler \cite{mahler}, who showed that the series 
$$ 
f(\zeta)=\sum_{n=0}^\infty \zeta^{2^n}, 
$$ 
which satisfies the functional equation 
$$ 
f(\zeta^2)=f(\zeta)-\zeta, 
$$ 
takes transcendental values at algebraic points $\al$ with $0<|\al|<1$. In particular, this 
includes the transcendence of the Kempner number and the other values of the series (\ref{kempner}) 
for integers $u\geq 2$. 

The analysis of the series (\ref{power}) by the methods of Loxton and van der Poorten, and 
a proof of the above conjecture, is an interesting challenge for the future. 

\subsection*{Acknowledgments} 
ANWH is funded by Fellowship EP/M004333/1 from the EPSRC,  
and grant IEC\textbackslash R3\textbackslash 193024 from the Royal Society. 
He thanks 
Evgeniy Zorin for pointing out the work of Loxton and van der Poorten. 
%and thanks the School of Mathematics \& Statistics, UNSW for 
%hospitality and support under the Distinguished Researcher Visitor Scheme.
%from the 
%Engineering and Physical Sciences Research Council. 
JLV is supported by grant PGC2018-096504-B-C32 from MINECO/FEDER. 

%%%%%%%%%%% To ease editing, use normal size for the references:

\normalsize

%----------------

%----------------

%----------------
\end{document}